\documentclass[12pt,reqno]{amsart}

\usepackage{graphicx}
\usepackage{mathrsfs}
\usepackage{amsthm}
\usepackage{amssymb}
\usepackage{amsmath,amsthm,amsfonts}
\usepackage{mathabx}
\usepackage{hhline}
\usepackage{textcomp}
\usepackage{amsmath,amscd}
\usepackage[all,cmtip]{xy}
\usepackage{mathtools}
\usepackage[margin=1in]{geometry}
\usepackage{lipsum}
\usepackage{extarrows}
\usepackage[displaymath]{lineno}
\usepackage[singlespacing]{setspace}%
\usepackage[colorlinks=true,breaklinks=true,bookmarks=true,urlcolor=blue,
citecolor=blue,linkcolor=blue,bookmarksopen=false,draft=false]{hyperref}

\providecommand{\U}[1]{\protect\rule{.1in}{.1in}}

\theoremstyle{definition}
\newtheorem{theorem}{Theorem}[section]
\newtheorem{lemma}{Lemma}[section]
\newtheorem{remk}{Remark}[section]

\newtheorem{defn}{Definition}[section]
\newtheorem*{theorem*}{Theorem}
\newtheorem{cor}{Corollary}[section]

\numberwithin{equation}{section}

\DeclareMathAlphabet{\mathpzc}{OT1}{pzc}{m}{it}

\newcommand{\rr}{\mathbb{R}}
\newcommand{\nn}{\mathbb{N}}

\newcommand{\sn}{\parallel}

\DeclareMathOperator{\dd}{D}

\DeclareMathOperator{\Lip}{Lip_{loc}}
\newcommand{\bb}{\mathcal{B}}
\begin{document}

\title{A Global Diffeomorphism Theorem for Fr\'{e}chet spaces}

\author{Kaveh Eftekharinasab}
\thanks{The author would like to thank the reviewer for the valuable comments  }
\address{Topology lab.  Institute of Mathematics of National Academy of Sciences of Ukraine, Tereshchenkivska st. 3, Kyiv, 01601 Ukraine}

\email{kaveh@imath.kiev.ua}



\date{}


\begin{abstract}
We give sufficient conditions for a $ C^1_c $-local diffeomorphism between Fr\'{e}chet spaces to be a global one.   
We extend the Clarke's theory of generalized gradients to the more general setting of Fr\'{e}chet spaces. As a consequence, we  define the Chang Palais-Smale condition for Lipschitz functions and show that a function which is bounded below and satisfies the Chang Palais-Smale condition at all levels is coercive. We prove a version of the mountain pass theorem for Lipschitz maps in the Fr\'{e}chet setting and show that along with the Chang Palais-Smale condition we can obtain a global diffeomorphism theorem.   
\end{abstract}

\maketitle
The problem of finding sufficient conditions for a local diffeomorphism to be a global one has been investigated
by many authors in the framework of Banach spaces, cf.~\cite{lenght} and references therein. 
But it has not been the subject of study for more general Fr\'{e}chet spaces. In~\cite{k1} we found sufficient conditions
 that indicate when smooth tame maps are global diffeomorphisms. The purpose of this paper is to find weakened conditions for $ C^1_{c} $-maps. To do this, we will apply the Clarke's theory of generalized gradients. By means of this theory the problem of global invertibility  of non-differentiable maps has been studied in Banach spaces by many authors cf.~\cite{john,4,ka}, but nothing exists for Fr\'{e}chet spaces. 
 
 The calculus of generalized gradients involves Lipschitz maps also on dual spaces weak$^{*}$ topology suffices. Thus, we may expect to carry it over to the Fr\'{e}chet setting without much difficulty. To this end, we start with the definition of the Clarke's subdifferential of Lipschitz functions and present some of its basic properties. We then naturally formulate the Palais-Smale condition in the sense of Chang~\cite{4}.  By means of Ekeland's variational
principle we prove that any lower bounded function that satisfies the Chang Palais-Smale condition at all levels is coercive.

As pointed out by Kartiel~\cite{ka}, mountain pass theorems can be used to obtain global homemorphism theorems. These  theorems  has many extensions and variations particularly, Shuzhong~\cite{shi} generalizes this theorem to locally Lipschitz functions on Banach spaces. Following his ideas  we prove the mountain pass theorem
for Fr\'{e}chet spaces, see Theorem~\eqref{mp}. The desired advantage of this theorem is that an obtained convergent subsequent satisfies the Chang Palais-Smale condition.

Finally, we prove the main theorem which roughly states that if $ \varphi $
is a $ C^1_c $-locally diffeomorphism of Fr\'{e}chet spaces and if for an appropriate coercive auxiliary function $ \imath $, a function
$x \mapsto \imath (\varphi (x) - y) $ for any $ y $ satisfies the Chang Palais-Smale condition then $ \varphi $ is a global diffeomorphism. 

It might not always be easy even for Banach spaces to check if a map satisfies the Chang Palais-Smale condition, therefore, for the Banach case another approach  which is based on the path lifiting property has been developed, see~Plastock~\cite{plas}. A potential line for further studies would be the generalization of this approach for Fr\'{e}chet spaces as well. 

Despite the fact that the theory of  Fr\'{e}chet spaces has a remarkable relation with both linear and non-linear problems  but not many methods for solving different types of differential equations are known. 
Our motivation here has an eye on future applications to ordinary differential equations. It is known that each global existence theorem for an autonomous system in Banach spaces has a correspondence with a global inversion theorem. Analogously, we would expect that such theorems will play notable role in the theory of differential equations in Fr\'{e}chet spaces.           
\section{Clarke's subdifferential}
In this section we extend some  basic concepts of the generalized gradients calculus to the Fr\'{e}chet setting. In most cases the proofs  have elementary calculus nature  and similar to their Banach analogues so we merely give references. 
  
We denote by $ F$ a Fr\'{e}chet space and by $ F' $ its dual. We assume that the topology of $ F $ is defined by
an increasing sequence of seminorms $ \sn \cdot \sn^1_F \leqq \sn \cdot \sn^2_F \leqq \cdots$. A translation-invariant complete metric inducing the same topology on $ F $ can then be defined by 
\begin{equation}\label{metric}
d_F (f,g) = \sum_{i=0}^{\infty} \dfrac{\sn f -g \sn^i_F}{1+\sn f -g \sn^i_F}.
\end{equation}
 A ball with  center $ x $ and  radius $ r $ in 
$ F $ and $ F' $ is denoted by $ B_r(x) $ and $ B'_r(x) $, respectively. The boundary of a set $ U $ is denoted by $ bd \, U $.
We will use the Keller's notion of $ C^k_c $-maps, see~\cite[Definition 2.2]{k2}.

 The weak topology $ \sigma (F,F') $ on $ F $ is given by the fundamental system of seminorms
\begin{equation*}
\rho_{\phi'}(x) \coloneq \sup_{y \in  \phi'}  \mid y(x) \mid,
\end{equation*}
where $ \phi' $ runs through the set $ \Phi' $ of finite subsets of the dual space $ F' $. The weak$^*$ topology $ \sigma (F',F) $ on $ F' $ is given by the fundamental system of seminorms
\begin{equation*}\label{we}
\rho_{\phi}(y) \coloneq \sup_{x \in  \phi}  \mid y(x) \mid,
\end{equation*}
where $ \phi $ runs through the set $ \Phi $ of  finite subsets of $ F $. Let $\langle .,.\rangle  $ be the dual pairing between $ F $ and $ F' $.

Let $ \Lip (F,\rr)$ be the set of all locally Lipschitz functions on $ F $ and $ \varphi \in \Lip (F,\rr)$. As in~\cite{1} we define for each $ f \in F $ the generalized directional  derivative, denoted by $ \varphi^{\circ}(f,g)  $, in the direction $ g \in F$ by
$$ \varphi^{\circ}(f,g)  \coloneq \limsup_{h \rightarrow f, t\downarrow 0}  \frac{\varphi(h+tg)-\varphi(h)}{t}, \quad t \in \rr, h \in F.$$
It can be easily seen that the function $ f\rightarrow  \varphi^{\circ}(f,g)$ is locally Lipschitz, positively homogeneous
and sub-additive. For any $ f \in F $ we define the Clarke's subdifferential of $ \varphi $, denoted by $ \partial \varphi $, as follows:
$$
\partial \varphi (f) \coloneq \Big\{ f' \in F' \mid (\forall g \in F) \langle f',g \rangle \leqq \varphi^{\circ}(f,g) \Big \}.
$$
\begin{lemma}\label{compact}
	\begin{itemize}
		\item [(a)] $ \varphi^{\circ}(f;g) $ is upper semi-continuous as a function of $ (f,g) $ and, as a function of $ g $ alone, is Lipschitz on $ F $.
		\item[(b)]$ \varphi^{\circ}(f,-g) = (-\varphi^{\circ})(f,g) $.
		\item [(c)] For every $ g \in F $, $ \varphi^{\circ} (f;g) = \max  \{ \langle h,g\rangle : h \in \partial \varphi(f) \}$.
		\item[(d)] \label{g}$ g \in \partial \varphi(f) $ if and only if $ \varphi^{\circ}(f;h)\geq \langle g,h\rangle, \, \forall h \in F $. 
		\item [(e)] \label{f} Suppose sequences $ (f_j) \subset F $ and $ (g_j) \subset F'$ are such that $ g_j \in \partial \varphi(f_j) $. If $ f_j \rightarrow f $ and 
		$ g $ is a cluster point of  $ (g_j) $, then $ g \in \partial (f) $.
		\item [(f)]$\partial(t \varphi)(f) =t \partial \varphi (f), \, \forall t \in \rr$.
		\item[(g)] If $ f $ is a local minimum of $ \varphi $, then $ 0 \in \partial \varphi(f) $.
		\item [(h)] $\partial (\varphi + \psi)(f) \subset  \partial \varphi (f)+ \partial \psi (f)$.
		
       \end{itemize}
\end{lemma}
\begin{proof}
		The proofs of $ (a)-(h) $ are easy and similar to the Banach case cf.~\cite[Prop. 2.1.1(b)]{2};\cite[Prop. 2.1.1(c)]{2};~\cite[Prop. 2.1.2(b)]{2};~\cite[Prop. 2.1.5(a)]{2};\cite[Prop. 2.1.5(b)]{2};~\cite[Prop. 2.3.1]{2}; 
		\cite[Prop. 2.3.2]{2}; \cite[Prop. 2.3.3]{2}, respectively.	
\end{proof}
\begin{lemma}
The subdifferential $ \partial \varphi (f) $  is a nonempty, convex and weak$^*$ compact subset of $ F' $.  
\end{lemma}
\begin{proof}
	 The Hahan-Banach theorem and Bourbaki-Alaoglu theorem are available for Fr\'{e}chet spaces, therefore,  it is enough to apply the arguments of Clarke~\cite[Prop. 2.1.2(a)]{2}.
\end{proof}
\begin{lemma}[Mean Value Theorem]\label{mean}
	Let $ f,g \in F $ and $ \varphi $ be a Lipschitz function on an open set containing the line segment $ [f,g] $. Then there exists
	$ \theta \in (0,1)  $ such that 
	\begin{equation}
	\varphi(g) - \varphi(f) \in \langle \partial \varphi (g + (\theta (g-f))), g-f\rangle.
	\end{equation}
\end{lemma}
\begin{proof}
The proof is very similar to that of~\cite[Theorem 2.3.7]{2}.
\end{proof}

\begin{lemma}[Chain rule]\label{chain}
	Let $ E $ be a Fr\'{e}chet space, $ \varphi : E\rightarrow F  $ a $ C^1_c $-map in a neighborhood of $ e \in E $, and $ \psi : F \rightarrow \rr $ a locally Lipschitz map.
	Then $ \tau = \psi \circ \varphi  $ is locally Lipschitz and $$ \partial \tau (e) \subseteqq  \partial \psi (\varphi (e)) \circ \dd \varphi (e).  $$ If  $ \psi $ is regular at $\varphi(e) $ then the equality holds.
\end{lemma}
\begin{proof}
The proof is also quite analogous to the Banach case~\cite[Theorem 2.3.10]{2}.
\end{proof}
We recall that a family $\bb$ of bounded subsets of $F$ that covers $F$
is called a bornology on $F$ if it is directed upwards by inclusion and if for every $B \in \bb$ and $r \in \rr$ there is
a $C \in \bb$ such that $r \cdot B \subset C$. 

Let $E$ be a Fr\'{e}chet space, $\bb$ a bornology on $F$ and $L_{\bb}(F,E)$  the space of all linear continuous maps from $F$ to $E$. 
The $\bb$-topology on $L_{\bb}(F,E)$ is a Hausdorff locally convex topology defined by all seminorms obtained as follows:
\begin{equation} \label{iq}
  \parallel L \parallel_{\bb} ^n \coloneq \sup \{ \parallel L (f) \parallel_E^n :  f \in B \, \&\, B \in \bb\}.
\end{equation} 
Suppose that $\bb$ consists of all compact sets, 
then the $\bb$-topology on the space  $L _{\bb}(F,\rr) = F_{\bb}'$ of all continuous linear functional on $F$, the dual
of $F$, is the topology of compact convergence. Let $U \subset F$ be open and  $\varphi : F \rightarrow E $ a Keller $C^1_{c}$-map at $x \in U$.
The derivative of $\varphi$ at $x$, $\dd \varphi(x)$, is an element of $F_{\bb}'$. We denote by $ < . , .>_{\bb}  $ the duality pairing between
$F$ and $F'_{\bb}$

\begin{lemma}
Let $\varphi : U \subset F \rightarrow \rr$ be Lipschitz in open neighbourhood $ U$ of $x$. If $\varphi$ is a $C^1_c$-map
at $x$, then $\dd \varphi (x) \in \partial \varphi (x)$.
\end{lemma}
\begin{proof}
By definition for all $h \in F$ we have  $$\dd \varphi(x)(h) = < \dd \varphi(x),h>_{\bb}.$$
By our definition of differentiability we get $ \dd \varphi(x)(h) \leqq \varphi^{\circ} (x,h)$, but $$ < \dd \varphi(x),h>_{\bb} \, \geqq \, < \dd \varphi(x),h> $$
therefore  $ < \dd \varphi(x),h>  \leqq \varphi^{\circ} (x,h)$ thereby by Lemma~\ref{compact}(d)  we obtain $\dd \varphi (x) \in \partial \varphi (x)$.
\end{proof}

\section{Chang Palais-Smale condition}
A point $ f \in F $ is called a critical point of $ \varphi $ if $ 0 \in \partial \varphi (f) $, 
that is $ \varphi^{\circ} (f;g) \geqq 0, \forall g \in F$. The value of a critical point is called
a critical value. 

We define for each $ \phi \in \Phi $ the function $ \lambda_{\varphi,\phi} $ on $F$ as follows
\begin{equation}
\lambda_{\varphi,\phi}(f) = 
\min_{y \in \partial \varphi(f)}  \rho_{\phi}(y).	
\end{equation}
The seminorms $\rho_{\phi} (\cdot)$ are bounded below and weak$^*$ lower semi-continuous because they arise as the pointwise supremum of the continuous absolute value function. Also, $ \partial \varphi (x) $ is weak$^*$-compact therefore the minimum is obtained.
\begin{lemma}\label{pis}
The set-valued mapping $ f \mapsto \partial \varphi{(f)} $ is locally bounded and weak$^{*}$ upper semi-continuous.
\end{lemma}
\begin{proof}
The proof is a slight modification of the Banach case, cf.~\cite[Theorem 1.1.2]{giles} 
\end{proof}
\begin{lemma} \label{smi}
For each $ \phi \in \Phi $ the function $ \lambda_{\varphi,\phi}(f) $ is sequentially lower semi-continuous
\end{lemma}	
\begin{proof}
	If $ \lambda_{\varphi,\phi}(f) $ is not sequentially lower-continuous there exist a sequence $ f_n \rightarrow f_0$ such that $\lim_{n \rightarrow \infty} \lambda_{\varphi,\phi}(f_n) < \lambda_{\varphi,\phi}(f_0) $. Let a sequence $y_n \in  \partial \varphi(f_n) $ be such that $ \rho_{\phi}(y_n) =\lambda_{\varphi,\phi}(f_n) $. By Lemma~\ref{pis}
	there exist a weak$^{*}$ open set $U$ in $F'$ such that $ \partial\varphi(f_0) \subseteq U $ and
	a neighborhood $ V $ of $ f_0 $ on which the mapping is bounded such that there exists
	a subsequence $ (f_{n_i}) $ of $ (f_n) $ in $ V $ so $ y_{n_i} \in \partial \varphi(f_{n_i}) $ and $ y_{n_i} \in U $.
	Since $ \{y_{n_i} \}$ is bounded it has a weak cluster $ y_0 $ and hence by Lemma~\ref{compact}(e) we have $ y_0 \in \partial \varphi(f_0) $ but
	$$
	\lambda_{\varphi,\phi}(f_0) \leqq \rho_{\phi}(y_0) \leqq
	 \liminf_{{n_i} \rightarrow \infty} \rho_{\phi} (y_{n_i}) 
	$$
	which is a contradiction. 
\end{proof}

\begin{defn}[Chang PS-condition]
	Let $ \varphi \in \Lip(F,\rr) $. We say that $ \varphi $ satisfies
	the Palais-Smale condition in the Chang's sense, Chang PS condition for short, if any sequence $ (f_n) $ in $ F $ such that $ \varphi(f_n) $ is bounded and for all
	$ \phi \in \Phi $
	\begin{equation}\label{1}
	\lim_{n \rightarrow \infty} \lambda_{\varphi,\phi}(f_n) = 0,
	\end{equation}
	possesses a convergent subsequence. Also, if any sequence $ (f_n) \subset F $ such that $ \varphi(f_n) \rightarrow c \in \rr$ and satisfies~\eqref{1} possesses a convergent subsequence we say that $ \varphi $ satisfies the Chang PS condition at level $ c $. 
\end{defn}
Suppose that $ \varphi \in \Lip(F,\rr) $  satisfies the Chang PS-condition. Let $ (f_n) $ be any sequence in $ F $ that converges to $ f_0 $ and satisfies~\eqref{1}. Since by Lemma~\ref{smi} 
the functions $ \lambda_{\varphi,\phi}(f_n)$ are sequentially lower semi-continuous
it follows that  $ \forall \phi \in \Phi $
\begin{equation*}
\lim_{n \rightarrow \infty} \lambda_{\varphi,\phi}(f_n) = \liminf_{n \rightarrow \infty} \lambda_{\varphi,\phi}(f_n) \geqq  \lambda_{\varphi,\phi}(f_0). 
\end{equation*}
Whence $ \lambda_{\varphi,\phi}(f_0)=0  $, that is the $zero$ function in $ F' $ belongs to $\partial \varphi(f_0) $, hence $ f_0 $ is a critical point. 

Now we prove that a functional $ \varphi \in \Lip(F,\rr) $ that satisfies the Chang PS condition  at all levels is coercive. 
The idea of proof is inspired by the work of Brezis and Nirenberg~\cite{6}.  

A functional $ \varphi : F \rightarrow \rr $ is said to be coercive if $ \varphi (f)  \rightarrow + \infty$ as $ \sn f \sn _F^1 \rightarrow \infty $.

We will need the following version of Ekeland's variational principle. 
\begin{theorem}\label{ekf}\cite{ek2}
	Let $(X, \sigma)$ be a complete metric space.
	Let a functional $f : X \rightarrow (-\infty, \infty]$ be semi-continuous, bounded
	from below and not identical to $+\infty$. 
	Then, for any $ \epsilon > 0 $ and every point $ x_0 \in X $ there exists $u \in X$ such that
	\begin{enumerate}
		\item $f(u) \leqq  f(x_0) - \varepsilon \sigma (u,x_0)$
		\item $f(u) \leqq f(x) + {\epsilon} \sigma (x,u), \quad \forall x \in X$.
	\end{enumerate}
\end{theorem}
\begin{theorem}\label{c1}
	Let $ \varphi \in \Lip(F,\rr) $ and let \[\alpha \coloneq \liminf_{\sn f \sn^1 \rightarrow \infty } \varphi(f)\]  be finite. Then there exists a sequence $ (f_n) \subset F $ such that $ \sn f_n \sn^i \rightarrow \infty (\forall i \in \nn) $, $ \varphi (f_n) \rightarrow \alpha,$ and $ \lambda_{\varphi,\phi}(f_n)\rightarrow 0 $ 
	for all $ \phi \in \Phi $.
\end{theorem}
\begin{proof}
	Define 
\begin{equation}\label{mr}
m(r) \coloneq \inf_{\sn f \sn^1 \geqq r}\varphi(f).
\end{equation}
	The function $ m(r) $ is a non-decreasing and 
	\begin{equation}\label{ee}
	\lim_{r \rightarrow \infty} m(r) = \alpha .
	\end{equation} 
	By~\eqref{ee} for each $ \varepsilon >0 $
	there exists $ r_1   $ such that for all $ r\geqq r_1 $
	\begin{equation} \label{27}
	 \alpha -\varepsilon^2 \leqq m(r).
	\end{equation}
	For a fixed $ \varepsilon >0 $ choose a number 
	\begin{equation}\label{46}
		r_2 \geqq \max \{ r_1,2\varepsilon\}.
	\end{equation}
	 By our assumption we can fix some $ z_0 $ with $  \sn z_0 \sn^1 \geqq 2r_2$ such that
\begin{equation}\label{42}
	\varphi(z_0) <  \alpha + \varepsilon^2
\end{equation} 
	Let $ \mathbf{F} = \{ f \in F :  \sn f \sn ^1 \geqq r_2 \}$. It is closed in $ F $, so it is a complete metric space by the induced  metric~\eqref{metric}. Moreover, $ \varphi $ is lower semi-continuous on $ F $ and so on $ \mathbf{F} $. Also, by
	\eqref{mr},\eqref{27} and \eqref{46}
	$$ \varphi(u)  \geqq m (\sn u \sn_F^1) \geqq \alpha - \varepsilon^2, \quad \forall u \in F \, \mathrm{with} \, \sn u \sn_F^1 \geqq r_2.$$
	So $ \varphi $ is lower bounded, and therefore, all assumptions of Theorem~\ref{ekf} are fulfilled for  
	$\mathbf{F} $. Thus, there is $ g \in \mathbf{F} $ such that 
	\begin{equation} \label{lklk}
	\varphi(g) \leqq  \varphi(x) + \varepsilon d_F(g,x), \quad \forall x \in \mathbf{F}
	\end{equation}
	\begin{equation}\label{24}
	\varphi (g) \leqq \varphi(z_0) - \varepsilon d_F(g,z_0)
	\end{equation}
	 It follows that \eqref{mr},\eqref{27}, \eqref{46},\eqref{24} and \eqref{42} 
	\begin{equation}
		\alpha - \varepsilon^2 \leqq m(r_2) \leqq \varphi(g) \leqq \varphi(z_0)-\varepsilon d_F(g,z_0) \leqq \alpha + \varepsilon^2 - \varepsilon d_F(g,z_0)  
	\end{equation}
	Hence 
	\begin{equation}
		d_F(g,z_0) \leqq 2 \varepsilon. 
	\end{equation}
	Thereby, by \eqref{46}
	\begin{equation}
	 d_F(g,0)   \geq d_F (z_0,0)  - d_F (g,z_0)  \geqq 2r_2 - 2 \varepsilon \geqq r_2.
	\end{equation}
	Whence  
	$ g $ is an interior point of $ \mathbf{F} $. Define on $ \mathbf{F} $ the function $$ \widetilde{\varphi}(h) \coloneq  d_F(g, h)  + \varphi(h).$$
	The function $\widetilde{\varphi}(h)$ attains its minimum in $ g \in \mathrm{Int} \,\mathbf{F}$ by virtue of~\eqref{lklk}. Therefore
	$$ 0 \in \partial \widetilde{\varphi}(g) \subseteq \partial \varphi (g) + \varepsilon B'_F$$ 
	 where $ B'_F $ is the closed unit ball in $ F' $. Thus,
	 $$ \lambda_{\varphi,\phi}(g) = \min \{ \rho_{\phi}(h)  \mid h \in \partial \varphi (h)\} \leqq \varepsilon. $$
	 Letting
	$ \varepsilon = \varepsilon_n \downarrow 0 $ completes the proof.
\end{proof}
\begin{cor}
	If $ \varphi \in \Lip (F,\rr) $ is bounded below and satisfies the Chang PS condition at $ c $ for all $ c \in \rr $, then it is coercive.
\end{cor}
\begin{proof}
	If it is not coercive then $ \alpha = \liminf_{\sn f \sn^1 \rightarrow \infty} \varphi (f)  $ is finite.
	Then by Theorem~\ref{c1} there exists a sequence $ (f_n) \subset F $ such that $ \sn f_n \sn^i \rightarrow \infty (\forall i \in \nn),\, \varphi(f_n) \rightarrow \alpha  $ and $ \lambda_{\varphi,\phi}(f_n)\rightarrow 0 $ 
	for all $ \phi \in \Phi $. Then the Chang PS-condition at $ \alpha $ yields that $ (f_n) $ has a convergent subsequent which is a contradiction. 
\end{proof}
\section{The mountain pass theorem }
Following the lines of the mountain pass theorem for Banach spaces due to Shuzhong~\cite{shi} we prove a version of the mountain pass theorem for locally Lipschitz functions between Fr\'{e}chet spaces. This is the most
suitable version for our goals as it involves the Chang PS condition.

Let $ \varphi \in \Lip(F, \rr) $ be a function.  
Let $ \mathcal{U} $ be an open neighborhood of $ zero $ and $ f \notin \overline{\mathcal{U}} $ be given such that for a real number $ m $
\begin{equation}\label{ine}
\max \{ \varphi(0), \varphi(f)\} < m \leqq \inf_{bd \, \mathcal{U}} \varphi.
\end{equation}
Let $$\Gamma \coloneq \{ \gamma \in C ([0,1]; F) : \gamma (0)=  0,\gamma(1)=f \}$$ be the  space of continuous paths joining $ 0 $ and $ f $. Consider
the Fr\'{e}chet space $C ([0,1]; F)$
 with the family of seminorms $$\sn \gamma \sn_{\Gamma}^{i'} = \sup_{t \in [0,1]} \sn \gamma (t) \sn^i_F.$$
 Let 
 \begin{equation}\label{m2}
 	d_{\Gamma} (\ell,\gamma) = \sum_{i'=0}^{\infty} \dfrac{\sn \ell -\gamma \sn^{i'}_{\Gamma}}{1+\sn \ell -\gamma \sn^{i'}_{\Gamma}}
 	\end{equation}
 be the metric that defines the same topology.	
 We can easily varify that $\Gamma$ is a closed subset of $C ([0,1]; F)$ so it is a complete metric space with
 the induced metric $ d_{\Gamma} $.
 
 In the sequel we will apply the following weak form of  Ekeland's variational principle.
 \begin{theorem}\label{ek}\cite[Theorem 1 bis.]{ik}
 	Let $(X, \sigma)$ be a complete metric space.
 	Let a functional $f : X \rightarrow (-\infty, \infty]$ be semi-continuous, bounded
 	from below and not identical to $+\infty$. 
 	Then, for any $ \epsilon > 0 $ there exists $x \in X$ such that
 	\begin{enumerate}
 		\item $f(x) < \inf_{X}f + \epsilon$
 		\item $f(x) \leqq f(y) + {\epsilon} \sigma (x,y), \quad \forall y \neq x \in X$.
 	\end{enumerate}
 \end{theorem}
The idea of the proof of the following mountain pass theorem is to define a function
$\Psi(\gamma) = \max_{[0,1]}\varphi(\gamma(t))$ on $C ([0,1]; F)$ and show that it is locally Lipschitz.
Then we find almost minimizers with some certain conditions by using Ekeland's variational principle. We pick a
sequence of these points on $\Gamma$ and associate it with a sequence on $F$ which satisfies the requirement of 
the Chang PS-condition for $ \varphi $. The limit of a subsequence of this sequence on $F$ is a critical point
of $ \varphi $.
\begin{theorem}\label{mp}
	Suppose $ \varphi \in \Lip(F,\rr) $ satisfies~\eqref{ine} for a real number $ m $. Let 
	\begin{equation}\label{c}
	c = \inf_{\gamma \in \Gamma} \max_{t \in [0,1]} \varphi (\gamma (t)) \geqq m. 
	\end{equation} 
	Then there exists a sequence $ (f_n) \subset F $	such that $ \varphi(f_n) \rightarrow c $ and satisfies~\eqref{1}.
	Moreover, if  $ \varphi $ satisfies the Chang PS condition then $ c $ is a critical value of $ \varphi $.
\end{theorem}
\begin{proof}
	Define the function $ \Psi : C ([0,1]; F) \rightarrow \rr $ by
	\begin{equation}
	\Psi(\gamma) = \max_{[0,1]}\varphi(\gamma(t)).
	\end{equation}
	Let $ \gamma \in C([0,1],F) $, for any $ t \in [0,1] $ there are positive numbers $ r_t, c_t  $ such that
	\begin{equation}
	\forall f_1\, ,f_2 \in B_{r_t} (\gamma (t)),\quad \mid \varphi (f_1) - \varphi(f_2) \mid \leqq c_t \sn f_1 - f_2 \sn^1_F .
	\end{equation}
	The family $ \{ B_{r_t}(\gamma(t))\}_{t \in [0,1]} $ is an open covering of the compact set $ \gamma([0,1]) $, therefore, there is a finite sub-covering  $ \{ B_{r_{t_j}}(\gamma(t_j))\}_{j = 1, \cdots, k }$ of $ \gamma([0,1]) $. Hence by the Lebesgue's number lemma there exists a positive number $ r $
	such that for any $ f \in \gamma([0,1]) $ there exists some $ 1 \leqslant j \leqslant k $ such that 
	$ B_r(f) \subset B_{r_{t_j}}(\gamma(t_j)) $. 
	
	Set $ c_{\gamma} \coloneq \max_{1 \leqslant j \leqslant k } c_{t_j}$. Therefore
	\begin{equation}
	\forall t \in [0,1],\, \forall f_1,f_2 \in B_r (\gamma (t)),\quad \mid \varphi (f_1) - \varphi (f_2) \mid \leqq c_{\gamma} \sn f_1-f_2 \sn ^1_F .
	\end{equation}
	If $ \gamma_1,\gamma_2  \in C([0,1],F)$ satisfy
	$$\sn \gamma_j - \gamma \sn_{\Gamma}^{i'}  < r(\forall i' \in \nn),\, j=0,1. $$
	Then $$
	\mid \Psi (\gamma_1)- \Psi(\gamma_2)\mid = \mid \max_{t \in [0,1]} \varphi (\gamma_1 (t))- \max_{t \in [0,1]} \varphi (\gamma_2)(t) \mid \leqq
	$$
	$$
 \max_{t \in [0,1]} \mid \varphi (\gamma_1(t))-\varphi (\gamma_2(t)) \mid  \leqq	c_{\gamma}\max_{t \in [0,1]} \sn \gamma_1(t) -\gamma_2(t) \sn ^1_F \leqq 
$$
$$
 c_{\gamma}\max_{t \in [0,1]} \sn \gamma_1(t) -\gamma_2(t) \sn ^i_F = c_{\gamma} \sn \gamma_1 - \gamma_2 \sn_{\Gamma}^{i'}, \forall i \in \nn
$$
Therefore $ \Psi $ is locally Lipschitz.
	
	Let $ \varepsilon_j $ be a sequence of positive numbers converging to $ zero  $ and $  (\eta_j ) \subset C ( [0,1];F) $  a sequence such that $ \sn \eta_j - \gamma \sn_{\Gamma}^{i'} (\forall i' \in \nn) \rightarrow 0 $ as $ j \rightarrow \infty $ and for $ \eta \in C([0,1];F) $ 
	$$
	\Psi^{\circ}(\gamma;\eta) = \lim_{j \rightarrow \infty}\dfrac{\Psi(\eta_j + \varepsilon_j\eta) - \Psi(\eta_j)}{\varepsilon_j}.
	$$
	Set $$ M(\gamma) \coloneq \{ s \in [0,1] \mid \varphi (\gamma(s)) = \Psi (\gamma) \}.$$ For any 
	$ s_j \in M (\eta_j + \varepsilon_j\eta), j = 1,2,\cdots $ it follows that 
	\begin{equation}
	\dfrac{\Psi(\eta_j + \varepsilon_j\eta) - \Psi(\eta_j)}{\varepsilon_j} \leqq \dfrac{\varphi(\eta_j(s_j) + \varepsilon_j\eta(s_j)) - \varphi(\eta_j(s_j))}{\varepsilon_j}.
	\end{equation}
	By the mean value theorem, there exist $ \epsilon_j \in (0,1) $ and $ x^*_j \in \partial \varphi (\eta_j(s_j) +  \epsilon_j\varepsilon_j \eta (s_j) )$ such that
	\begin{equation}
	\dfrac{\varphi(\eta_j(s_j) + \varepsilon_j\eta(s_j)) - \varphi(\eta_j(s_j))}{\varepsilon_j} =  \langle x^*_j , \eta (s_j) \rangle , \quad j=1,2,\cdots
	\end{equation}
	The sequence $ ( s_j) $ has a convergent sequence, denoted again by   $ ( s_j) $, suppose that $ s_j \rightarrow s $.
	Then $ \eta_j(s_j) + \epsilon_j\varepsilon_j \eta (s_j) \rightarrow \gamma (s)$. By  Lemma~\ref{compact}(e) the sequence $ ( x^*_j) $ has a $ w^* $-cluster point
	$ x^* \in \partial \varphi (\gamma (s)) $. So we have $ \langle x_j^{*} ,\eta(s)\rangle \rightarrow \langle x^{*} ,\eta(s)\rangle$ and then 
	\begin{equation}
	\Psi^{\circ}(\gamma;\eta)\leqq \lim_{j \rightarrow \infty}\langle x_j^{*} ,\eta(s_j)\rangle \leqq \lim_{j \rightarrow \infty}\langle x_j^{*} ,\eta(s_j) - \eta(s)\rangle + \lim_{j \rightarrow \infty}\langle x_j^{*} ,\eta(s)\rangle.
	\end{equation}
	Since $ s_j \in M(\eta_j+\varepsilon_j\eta) $, we have
	$$
	\varphi (\eta_j(s_j) + \varepsilon_j \eta(s_j)) \geqq \varphi(\eta_j(t)+\varepsilon_j\eta(t)), \quad \forall t \in [0,1].
	$$
	Letting $ t \rightarrow \infty  $ yields $$ \varphi(\gamma(s))\geqq \varphi(\gamma(t)),  \forall t \in [0,1] $$  and hence $ s \in M(\gamma) $, 
	therefore 
	\begin{equation}\label{key}
	\Psi^{\circ}(\gamma;\eta)\leqq \max_{s \in M(\gamma)}\varphi^{\circ}(\gamma(s);\eta(s)), \, \forall \eta \in C([0,1],F).
	\end{equation}
	Set $$ C_{0}([0,1],F) \coloneq \Big \{\eta \in C([0,1],F); \forall t \in \{0,1 \}, \eta(t)=0  \Big \}.$$ Suppose for some $ \gamma \in C([0,1],F) $ we have $ M(\gamma) \subset (0,1) $ and there exists $ \varepsilon > 0 $ such that for $ \eta \in C_{0}([0,1],F) $
	\begin{equation}
	\Psi^{\circ}(\gamma;\eta) \geqq -\varepsilon \sn \eta \sn_{\Gamma}^{i'}(\forall i' \in \nn).
	\end{equation}
	We prove that there exists $ s \in M (\gamma) $ such that $ \forall h \in F $
	\begin{equation}\label{d}
	\varphi^{\circ}(\gamma(s);h) \geqq -\varepsilon \sn h \sn^i_F (\forall i \in \nn).
	\end{equation}
	If there there is no such $ s $ then for any $ t \in M(\gamma) $ there exits $ h_t \in F $ with $ \sn h_t \sn^i_F  =1 (i \in \nn)$
	such that $ \varphi^{\circ}(\gamma(t);h_t) < -\varepsilon $. The continuity of $ \gamma $ and the upper semi-continuity of
	$ \varphi^{\circ} $ implies that for any $ t \in M(\gamma) $ there exits $ h_t \in F $ with $ \sn h_t \sn^i_F =1 (\forall i \in \nn) $
	and $ \varepsilon_t' >0 $ such that
	\begin{equation}
	\varphi^{\circ}(\gamma (s);h_t) < - \varepsilon, \, \forall s \in B_{\varepsilon_t'}(t)= \{ s \in [0,1]; \mid s-t \mid < \varepsilon_t '\}.
	\end{equation}
	The family $ \{ B_{\varepsilon'_t}(t)\}_{t \in M(\gamma)} $ is an open covering of $ M(\gamma) $. Since $ M(\gamma) $ is compact there exist 
	$ t_1,\cdots ,t_k \in M(\gamma) $ such that $ M(\gamma)\subset \bigcup_{j=1}^{j=k}B_{\varepsilon_{t_j}'}(t_j) $. Since $ M(\gamma) $
	is a subset of $ (0,1) $ it follows that $ B_{\varepsilon_t'}(t) $ does not contain $ \{ 0,1\} $ for all $ t \in M(\gamma) $. Thereby
	\begin{equation}\label{sum}
	\{\bigcup_{j=1}^{j=k}B_{\varepsilon_{t_j}'}(t_j)\}\cup \{ [0,1] \setminus M(\gamma)\} = [0,1].
	\end{equation}
	Define 
	\begin{equation}
	\mu(t) = \dfrac{\Sigma_{j=1}^{j=k} h_{t_j}d_j(t)}{\Sigma_{j=0}^{j=k}d_j(t)}.
	\end{equation}
	Where
	\begin{equation}
	d_0 (t) = \min_{s \in M(\gamma) } \mid t-s \mid,\quad t \in [0,1],
	\end{equation}
	and 
	\begin{equation}
	d_j (t) = \min_{s \in [0,1] \setminus B_{\varepsilon_{t_j}'}(t_j) } \mid t-s \mid,\quad t \in [0,1], \, j=1,\cdots,k.
	\end{equation}
	By~\eqref{sum} it follows that $\Sigma_{j=0}^{j=k}d_j(t) >0$.
	
	By the above arguments we obtain $ \eta_{0} \in C_{0}([0,1];F) $ with $ \sn \eta_{0} \sn_{\Gamma}^{i'} \leqq 1 $.
	Since $ g \mapsto \varphi^{\circ} (f,g) $ is sublinear in $ g $,
	\begin{equation}
	\varphi^{\circ}(\gamma(t);\eta_{0}(t)) \leqq \dfrac{\Sigma_{j=1}^k d_j(t) \varphi^{\circ}(\gamma(t); h_{t_j})}{\Sigma_{j=0}^k d_j(t)}.
	\end{equation}
	Then for any $t \in M(\gamma)  $ we get 
	\begin{equation}
	d_0(t) = 0, d_j(t)>0 (j\neq 0) \Rightarrow \varphi^{\circ}(\gamma(t);h_{t_j}) < -\varepsilon.
	\end{equation}
	Therefore by Equation~\eqref{key} we have 
	\begin{equation*}
	\Psi^{\circ}(\gamma;\eta_0)\leqq \max_{s \in M(\gamma)}\varphi^{\circ}(\gamma(s);\eta_0(s)) < -\varepsilon \sn \eta_0 \sn_{\Gamma}^{i'}
	\end{equation*}
	which is a contradiction.
	
	Since $ \mathcal{U} $ separates 0 and $ f $, for any $ \gamma \in \Gamma $,
	\begin{equation}
	\gamma([0,1]) \bigcap bd \,\mathcal{U} \neq \emptyset.
	\end{equation} 
	Then by the assumptions of the theorem 
	\begin{equation}\label{b}
	\max_{t \in [0,1]} \varphi( \gamma(t)) \geqq \inf_{bd \, \mathcal{U}} \varphi \geqq m > \max \{ \varphi (\gamma(0)), \varphi (\gamma(1))\}
	\end{equation}
	Therefore for every $ \gamma \in \Gamma $,
	\begin{equation}
	M(\gamma)= \{s \in [0,1]; \varphi(\gamma(s)) = \max_{t \in [0,1]} \varphi(\gamma(t)) \} \subset (0,1).
	\end{equation}
	The restriction of $ \Psi $ to $ \Gamma $ is again locally Lipschitz and by~\eqref{b} and~\eqref{c} is bounded from below. Let $ (\alpha_n) $ be a sequence of positive numbers converging to $ zero $. By Ekeland's variational principle~\ref{ek} there exists a sequence $ (\gamma_n) \subset \Gamma $ such that 
	\begin{equation}
	c \leqq \Psi(\gamma_n) \leqq c+ \alpha_n
	\end{equation}
	and 
	\begin{equation}
	\Psi(\varrho)>\Psi(\gamma_n)-\alpha_n d_{\Gamma}(\varrho ,\gamma_n), \, \varrho \neq \gamma_n , n =1,2,\cdots
	\end{equation}
	Therefore for any $ \eta \in C_{0}([0,1];F) $ we obtain 
	\begin{equation}
	\Psi^{\circ}(\gamma_n;\eta) \geqq \limsup_{t \rightarrow 0}\dfrac{\Psi(\gamma_n+t\eta)-\Psi(\gamma_n)}{t}\geqq -\alpha_n \sn \eta \sn_{\Gamma}^{i'}, \, n=1,2,\cdots 
	\end{equation}
	Then by Equation~\eqref{d}, there exists $ s_n \in M(\gamma_n) $ such that $ \varphi (\gamma_n(s_n) ) = \Psi(\gamma_n)$ and
	\begin{equation}
	\varphi^{\circ}(\gamma_n(s_n);h) \geqq -\alpha_n \sn h \sn^i_F (\forall i \in \nn), \forall h \in F, \, n=1,2,\cdots 
	\end{equation}
	Let $ f_n = \gamma_n (s_n)$ for $ n = 1,2,\cdots $, then $ (f_n)$ is the desired sequence and  we have $ \varphi(f_n) \rightarrow c $, moreover
	\begin{equation}
	0 \in \partial \varphi(f_n)+\alpha_n \overline{U^{\circ}},
	\end{equation}
	where $ U $ is a neighborhood of the origin in $ F $. 
By the Chang PS condition $(f_n)$ has a convergent subsequent, denoted again by $(f_n)$, with the limit $ z $. Then,
$$
\varphi (z) = \lim_{n \rightarrow \infty } \varphi (f_n) = c
$$
and $ \forall w \in F \&\,i\in \nn $
$$
\varphi^{\circ} (z;w) \geqq \limsup_{n \rightarrow \infty}\varphi^{\circ}(\gamma_n (s_n); w) \geqq -\lim _{n \rightarrow \infty}
\alpha_n \sn w \sn^i_F = 0.
$$
That is $ 0 \in \partial \varphi(z) $.
\end{proof}

\section{A global diffeomorphism theorem}

In this section we apply the mountain pass theorem of the previous section to obtain
a  global diffeomorphism theorem.  

\begin{lemma} \label{lem}
Let $ \varphi \in \Lip (F, \rr)$ and bounded from bellow. Then there exists a sequence $ (f_n) $ such that 
$ \lim_{n \rightarrow \infty} \varphi(f_n) =\inf_{F} \varphi $ and for all
$ \phi \in \Phi $
\begin{equation}
\lim_{n \rightarrow \infty} \lambda_{\varphi,\phi}(f_n) = 0.
\end{equation}
\end{lemma}
\begin{proof}
Consider a sequence of positive nuumbers $ (\epsilon_n) $ converging to $ zero $. The function $\varphi  $
satisfies all assumptions of Ekeland variational principle~\ref{ek}, so we can find a sequence $ (f_n) $ such
that  $$ \varphi(f_n) < \inf_{F}\varphi + \epsilon_n $$ and 
\begin{equation}
\varphi(f) \geqq \varphi(f_n) - \epsilon_n d_{F}(f,f_n) , \forall f \neq f_n \in F.
\end{equation}
Assume $ f = f_n + t(g-f_n)  $ for some $ g \in F $ and a positive number $ t $, then we obtain
\begin{equation}
\dfrac{\varphi(f_n+t(g-f_n))-\varphi(f_n)}{t} \geqq -\epsilon_n d_F(g,f_n).
\end{equation}
Thus, for all $ g \in F $ if we let $ t \rightarrow 0 $ then
\begin{equation}
\varphi^{\circ}(f_n; g-f_n)  \geqq -\epsilon_n d_F(g,f_n).
\end{equation}
For a fixed $ f_n $ define the sets $$ \Theta_n \coloneq \{(h,t) \mid h \in F ; t > \varphi^{\circ}(f_n;h) \}$$
and $$ \Pi_{n} \coloneq \{ (h,t) \mid h\in F ; t < -\epsilon_n \sn h \sn^i_F (\forall i \in \nn) \}.$$ They are open convex sets with 
empty intersection so by Hahn-Banach separation theorem there exists a separating hyperplane determined by a functional $ \upsilon_n (h,t) = w_n(h)+ \alpha t $ for some $ \alpha \neq 0$, where
$ w_n $ is a linear functional on $ F $ such that $ w_n(0)=0 $. Let $w^{*}(h) = \dfrac{-1}{\alpha} w_n(h)$, then
$ \upsilon_n(h,  w^{*}(h)) =0 , \, \forall h \in F$. Thereby $  w^{*}(h) \leqq \varphi^{\circ}(f_n;h) $ hence by~\eqref{g}(g) we have $ w_n^{*}\in \partial \varphi(f_n) $. On the other hand $ \mid  w^{*}_n(h)  \mid \leqq \epsilon_n \sn h \sn^i (\forall i \in \nn) $ so for all $ \phi $ we have $ \lambda_{\varphi,\phi}(f_n) \leqq \rho_{\phi}(w^{*}) \leqq \epsilon_n$.
 Letting
$ \varepsilon_n \downarrow 0 $ completes the proof.
\end{proof}

\begin{theorem}\label{ch}
Let $ \imath : F \rightarrow [0, \infty )$ be a coercive locally Lipschitz function having the following two properties:
\begin{itemize}
\item $\imath(x)=0$ if and only if $x=0$,
\item $0 \in \partial \imath(y)$ if and only if $y =0$.
\end{itemize}
Further let $ \tau : E \rightarrow F$ be a local $ C^1_c $-diffeomorphism.  Suppose that for each $ f \in F $ the function $\jmath : F \rightarrow [0,\infty)$ given by
	\begin{gather*}
	\jmath (e) = \imath (\tau(e) - f ) 
	\end{gather*}
	satisfies the Chang Palais-Smale condition. Then $ \tau $ is a global diffeomorphism. 
\end{theorem}
\begin{proof}
We need to show that $ \tau $ is surjective and bijective. Let $ e_1 \neq e_2 \in E $ if $ \tau(e_1)\neq \tau(e_2) $ we have nothing to prove. Assume $ \tau(e_1)=\tau(e_2)=l $. Since $ \tau $ is a local diffeomorphism it follows that it is an open map,
therefore, there exist $ \sigma ,\alpha > 0 $  such that 

\begin{equation}\label{s}
B_{\alpha r}(l) \subset \tau (B_r(e_1)), \, \forall r \in (0, \sigma).
\end{equation}

Let $ \mathbf{r} \in (0,\sigma) $ be the smallest number  such that $ e_2 \notin \overline {B_{\mathbf{r}}(e_1)} $. Consider the function $ \jmath (e) = \imath (\tau(e) - l) $, therefore $ \jmath (e_1) = \jmath(e_2) = 0 $.

 Without the loose of generality we can suppose $ e_1 = 0$. By~\eqref{s} for $ e \in bd \, B_{\mathbf{r}}(0) $ we have $ 0 < m \leqq \jmath (e) $.
 Thus, all conditions of Theorem~\ref{mp} hold so there exists $ ( e_n ) \subset E $ such that 
 $ \lim_{n \rightarrow \infty} \jmath (e_n) = c $ for
 some $ c \geqq m  $ characterized by~\eqref{c}.
  Since $ \jmath (e_n) $ satisfies the Chang PS condition it has a convergent subsequent, denoted again by $ (e_n) $,  with the limit $ h $. Therefore, $ h $ is a critical point so $ 0 \in \partial \jmath (h) $ and $ \tau(h) \neq l $ since 
  $ \lim_{n\rightarrow \infty} \jmath (e_n) = \jmath(h) = c  \geqq m >  0$. 

By the chain rule~\eqref{chain} we have $ \partial \jmath (h) \subset \partial \imath (\tau (h) -l)  \circ \dd \tau (h) $. 
Therefore, there exists $v \in \partial \imath (\tau (h) -l)$ such that $ 0 = v \circ  \dd \tau (h)$. Since $ \tau $ is
a local diffeomorphism it follows that $ v = 0 $. Therefore, by our assumption on $ \imath,\, \tau(h)-l $ must be $ zero $ which is a
contradiction.  

Let $ g \in F $ be given and consider the function $ \jmath (e) = \imath (\tau(e)-g) $. By Lemma~\ref{lem} there exists a sequence $ (f_n) $ such that 
 $ \lim_{n \rightarrow \infty} \jmath (f_n) =\inf_{E} \jmath $ and for all
 $ \phi \in \Phi $ we have
 \begin{equation}
 \lim_{n \rightarrow \infty} \lambda_{\jmath,\phi}(f_n) = 0.
 \end{equation}
 Since $ \jmath $ satisfies the Chang PS condition the sequence $ (f_n) $ has a convergent subsequent, denoted again by 
 $(f_n)$,
 with the limit $ p $ which is a critical point of $ \jmath $ so $ 0 \in \partial \jmath (p) $. By the chain rule~\eqref{chain} we have $ \partial \jmath (p) \subset \partial \imath (\tau (p) -g)  \circ \dd \tau (p) $. Thus, there exists $ \xi \in \partial \imath (\tau (p) -g) $ such that $ 0 = \xi \circ  \dd \tau (p)$. Since $ \dd \tau $ is invertible at $ p $ we have $ \xi =0 $. Therefore by our assumption on $ \imath, \, \tau(p)=g $.  
\end{proof}
\begin{remk}
In~\cite{gutu} the analogue of the theorem for Banach spaces is obtained, where the applied auxiliary function  
is $ \dfrac{1}{2} \sn \cdot \sn^2 $ and it satisfies the weighted Chang PS condition. The results of~\cite{gutu} may
 also work with the type auxiliary function that we use. Nevertheless, we may attempt to extend Theorem~\eqref{ch} for 
 auxiliary functions that satisfy the weighted Chang PS condition which of course requires an appropriate mountain pass theorem. 
 \end{remk}
\bibliographystyle{amsplain}

\begin{thebibliography}{10}
\bibitem{6}	
H.~Brezis and L.~Nirenberg, Remarks on finding critical points, Comm.~Pure
Appl.~Math., 44 (1991), 939-963.

\bibitem{1}
F.~Clarke, Generalized gradients and applications, Trans.~Amer.~Math.~Soc., 205 (1975), 247-262.
\bibitem{2}
F.~Clarke, Optimization and nonsmooth analysis, Canadian Mathematical Society
Series of Monographs and Advanced Texts, John Wiley and Sons, Inc., New York, 1983.
\bibitem{3}
F.~Clarke, A new approach to Lagrange multipliers. Math.~Oper.~Res., 1 (1976),
165-174.
\bibitem{4}
K.~Chang. Variational methods for non-differentiable functionals and their applications to partial differential equations,
 J.~Math.~Anal.~and Appl., 80 (1981), 102-129.
\bibitem{k1}
K.~Eftekharinasab, On the existence of a global diffeomorphism between Fr\'{e}chet spaces, submitted 
\bibitem{k2}
K.~Eftekharinasab, Sard’s theorem for mappings between Fr\'{e}chet manifolds, Ukrainian Math. J.,
62 (2010), 1634-1641.
\bibitem{ik}
I.~Ekeland, Nonconvex minimization problems, Bulletin of the American Mathematical Society, Vol. 1, No. 3 (1979), 443-474.
\bibitem{ek2}
I.~Ekeland, On the variational principle, J. Math. Anal. Applic., 47 (1974),  324-353.
\bibitem{john}
F.~John, On quasi-isometric maps I. Comm.~Pure Appl.~Math., 21 (1968), 77-110. 
\bibitem{giles}
J.~R.~Giles, A Survey of Clarke’s Subdifferential and the Differentiability of Locally Lipschitz Functions, In: Progress in Optimization. Applied Optimization, vol. 30, Springer, Boston, MA, 1999. 
\bibitem{gutu}
O.~Gut\'{u}, Chang Palais-Smale condition and global inversion, Bull.~Math.~Soc.~Sci.~Math.~ Roumanie Tome 61 (109) No. 3 (2018), 293-303.
\bibitem{ka}
 G.~Katriel, Mountain-pass theorems and global homeomorphism theorems, Ann.~Inst.~Henri
Poincar\'{e}, Analyse Non Lin\'{e}aire, Vol.~11, No.~2 (1994), 189-209.
\bibitem{lenght} 
I.~ Garrido, O.~Gut\'{u}, J.~Jarmaillio, Global inversion and covering maps on length spaces, Nonlinear Analysis,
Vol. 73, No. 5, (2010), 1364-1374.
\bibitem{plas}
R.~Plastock, Homemorphisms between Banach spaces, Trans.~Am.~ Math.~Soc. 200 (1974) 169-183. 
\bibitem{shi}
S.~Shi, Ekeland's variational principle and the mountain pass lemma. Acta.~Math.~Sin.,
(N.S.), 1, No. 4, (1985), 348-355.

\end{thebibliography}

\end{document}